\documentclass[preprint,12pt]{elsarticle}

\usepackage{amsthm,amsmath,amssymb}
\usepackage{graphicx}
\usepackage{enumitem}
\usepackage{etoolbox}
\theoremstyle{plain}
\newtheorem{theorem}{Theorem}
\newtheorem{lemma}[theorem]{Lemma}

\newtheorem{proposition}[theorem]{Proposition}

\theoremstyle{definition}

\theoremstyle{remark}
\newtheorem{remark}[theorem]{Remark}

\begin{document}

\begin{frontmatter}



\title{On the structure of spikes}


\author{Vahid Ghorbani, Ghodratollah Azadi and Habib Azanchiler}

\address{Department of mathematics, University of Urmia, Iran}

\begin{abstract}
Spikes are an important class of 3-connected matroids. For an integer $r\geq 3$, there is a unique binary r-spike denoted by $Z_{r}$. When a circuit-hyperplane of $Z_{r}$ is relaxed, we obtain another spike and repeating this procedure will produce other non-binary spikes. The $es$-splitting operation on a binary spike of rank $r$, may not yield a spike. In this paper, we give a necessary and sufficient condition for the $es$-splitting operation to construct $Z_{r+1}$ directly from $Z_{r}$. Indeed, all binary spikes and many of non-binary spikes of each rank can be derived from the spike $Z_{3}$ by a sequence of The $es$-splitting operations and circuit-hyperplane relaxations.
\end{abstract}
\end{frontmatter}

\section{Introduction}
Azanchiler \cite{1}, \cite{2} extended the notion of $n$-line splitting operation from graphs to binary matroids. He characterized the $n$-line splitting operation of graphs in terms of cycles of the respective graph and then extended this operation to binary matroids as follows. Let $M$ be a binary matroid on a set $E$ and let $X$ be a subset of $E$ with $e\in X$. Suppose $A$ is a matrix that represents $M$ over $GF(2)$. Let $A^e_{X}$ be a matrix obtained from $A$ by adjoining an extra row to $A$ with this row being zero everywhere except in the columns corresponding to the elements of $X$ where it takes the value 1, and then adjoining two columns labeled $\alpha$ and $\gamma$ to the resulting matrix such that the column labeled $\alpha$ is zero everywhere except in the last row where it takes the value 1, and $\gamma$ is the sum of the two column vectors corresponding to the elements $\alpha$ and $e$. The vector matroid of the matrix $A^e_{X}$ is denoted by $M^e_{X}$. The transition from $M$ to $M^e_{X}$ is called an \emph{$es$-splitting operation}. We call the matroid $M^e_{X}$ as \emph{$es$-splitting matroid}.

Let $M$ be a matroid and $X\subseteq E(M)$, a circuit $C$ of $M$ is called an \emph{$OX$-circuit} if C contains an odd number of elements of $X$, and $C$ is an \emph{$EX$-circuit} if C contains an even number of elements of $X$. 
The following proposition characterizes the circuits of the matroid $M^e_{X}$ in terms of the circuits of the matroid $M$.
\begin{proposition} \textnormal{\cite{1}} Let $M=(E,\mathcal C)$ be a binary matroid together with the collection of circuits $\mathcal C$. Suppose $X\subseteq E$, $e\in X$ and $\alpha,\gamma\notin E$. Then $M^e_{X}=(E\cup\{\alpha,\gamma\},\mathcal C')$ where $\mathcal C'=(\cup_{i=0}^{5}\mathcal{C}_{i})\cup \Lambda$ with $\Lambda=\{e,\alpha,\gamma\}$ and\\
$\mathcal C_{0}=\{C\in \mathcal C: \text{$C$ is an $EX$-circuit}\}$;\\
$\mathcal C_{1}=\{C\cup \{\alpha\}: \text{$C\in \mathcal C$ and $C$ is an $OX$-circuit}\}$;\\
$\mathcal C_{2}=\{C\cup \{e,\gamma\}: \text{$C\in \mathcal C$, $e\notin C$ and $C$ is an $OX$-circuit}\}$;\\
$\mathcal C_{3}=\{(C\setminus e)\cup \{\gamma\}: \text{$C\in \mathcal C$, $e\in C$ and $C$ is an $OX$-circuit}\}$;\\
$\mathcal C_{4}=\{(C\setminus e)\cup \{\alpha,\gamma\}: \text{$C\in \mathcal C$, $e\in C$ and $C$ is an $EX$-circuit}\}$;\\
$\mathcal C_{5}=$ The set of minimal members of $\{C_{1}\cup C_{2}: C_{1},C_{2}\in \mathcal C,C_{1}\cap C_{2}=\emptyset$\\ \text{and each of $C_{1}$ and $C_{2}$ is an $OX$-circuit}$\}$.
\label{pro1}
\end{proposition}
It is observed that the $es$-splitting of a 3-connected binary matroid may not yield a 3-connected binary matroid. The following result, provide a sufficient condition under which the $es$-splitting operation on a 3-connected binary matroid yields a 3-connected binary matroid.

\begin{proposition}\textnormal{\cite{4}}
Let $M$ be a 3-connected binary matroid, $X\subseteq E(M)$ and $e\in X$. Suppose that $M$ has an $OX$-circuit not containing $e$. Then $M^e_{X}$ is a 3-connected binary matroid.
\label{pro2}
\end{proposition}  
To define rank-r spikes, let $E=\{x_{1},x_{2},...,x_{r},y_{1},y_{2}...,y_{r},t\}$ for some $r\geq 3$. Let $\mathcal C_{1}=\{\{t,x_i,y_{i}\}:1\leq i\leq r\}$ and $\mathcal C_{2}=\{\{x_{i},y_{i},x_{j},y_{j}: 1\leq i < j \leq r \}$. The set of circuits of every spike on $E$ includes $\mathcal C_{1}\cup \mathcal C_{2}$. Let $C_{3}$ be a, possibly empty, subsets of $\{\{z_{1},z_{2},...z_{r}\}: \text{$z_{i}$ is in $\{x_{i},y_{i}\}$ for all $i$}\}$ such that no two members of $\mathcal C_{3}$ have more than $r-2$ common elements. Finally, let $\mathcal C_{4}$ be the collection of all $(r+1)$-element subsets of $E$ that contain no member of $\mathcal C_{1}\cup \mathcal C_{2}\cup \mathcal C_{3}$.
\begin {proposition} \textnormal{\cite{3}} There is a rank-r matroid $M$ on $E$ whose collection $\mathcal C$ of circuits is 
$\mathcal C_{1}\cup \mathcal C_{2}\cup \mathcal C_{3}\cup \mathcal C_{4}$.
\label{pro3}
\end{proposition}
The matroid $M$ on $E$ with collection $\mathcal C$ of circuits in the last proposition is called a \emph{rank-$r$ spike with tip $t$} and legs $L_{1},L_{2},...L_{r}$ where $L_{i}=\{t,x_{i},y_{i}\}$ for all $i$.
In the construction of a spike, if $\mathcal C_{3}$ is empty, the corresponding spike is called
the \emph{rank-r free spike with tip $t$}. In an arbitrary spike M, each circuit in $\mathcal C_{3}$ is also a
hyperplane of $M$. Evidently, when such a circuit-hyperplane is relaxed, we obtain another spike. Repeating this procedure until all of the circuit-hyperplanes in $\mathcal C_{3}$ have been relaxed will produce the free spike. Now let $J_{r}$ and $\textbf{1}$ be the $r\times r$ and $r\times 1$ matrices of all ones. For $r\geq 3$, let $A_{r}$ be the $r\times(2r+1)$ matrix $[I_{r}|J_{r}-I_{r}|\textbf{1}]$ over $GF(2)$ whose columns are labeled, in order, $x_{1},x_{2},...,x_{r},y_{1},y_{2}...,y_{r},t$. The vector matroid $M[A_{r}]$ of this matrix is called the \emph{rank-r binary spike with tip t} and denoted by $Z_{r}$. Oxley \cite{3} showed that all rank-$r$, 3-connected binary matroids without a 4-wheel minor can be obtained from a binary $r$-spike by deleting at most two elements.
\section{Circuits of $Z_r$}
In this section, we characterize the collection of circuits of $Z_{r}$. To do this, we use the next well-known theorem.
\begin{theorem} \textnormal{\cite{3}}
A matroid $M$ is binary if and only if for every two distinct circuits $C_{1}$ and $C_{2}$ of $M$, their symmetric difference, $C_{1}\Delta C_{2}$, contains a circuit of $M$. 
\label{thm4}
\end{theorem} 
Now let $M=(E,\mathcal C)$ be a binary matroid on the set $E$ together with the set $\mathcal C$ of circuits where $E=\{x_{1},x_{2},...,x_{r},y_{1},y_{2}...,y_{r},t\}$ for some $r\geq 3$. Suppose $Y=\{y_{1},y_{2}...,y_{r}\}$. For $k$ in $\{1,2,3,4\}$, we define $\varphi_{k}$ as follows.\\
$\mathcal \varphi_{1}=\{L_{i}=\{t,x_i,y_{i}\}:1\leq i\leq r\}$;\\
$\mathcal \varphi_{2}=\{\{x_{i},y_{i},x_{j},y_{j}\}: 1\leq i < j \leq r \}$;\\
$\mathcal \varphi_{3}=\{Z\subseteq E: \text{$|Z|=r$, $|Z\cap Y|$ is odd and $|Z \cap\{y_{i},x_{i}\}|=1$ where $1\leq i\leq r$}\}$; and\\
\small{\[ \varphi_{4}=\begin{cases}
\{E-C: C\in \mathcal \varphi_{3}\},&\textnormal {if $r$ is odd};\\
\{(E-C)\Delta\{x_{r-1},y_{r-1}\}: {C\in \mathcal \varphi_{3}}\},& \textnormal {if $r$ is even}.
\end{cases}\]}

\begin{theorem} A matroid whose collection $\mathcal C$ of circuits is $\varphi_{1}\cup\varphi_{2}\cup\varphi_{3}\cup\varphi_{4}$, is the rank-r binary spike. 
\label{thm5}
\end{theorem}
\begin{proof}
Let $M$ be a matroid on the set $E=\{x_{1},x_{2},...,x_{r},y_{1},y_{2}...,y_{r},t\}$ such that $\mathcal C(M)=\varphi_{1}\cup\varphi_{2}\cup\varphi_{3}\cup\varphi_{4}$. Suppose $Y=\{y_{1},y_{2},...y_{r}\}$. Then, for every two distinct circuits $C_{1}$ and $C_{2}$ of $\varphi_{3}$, we have $C_{1}\cap Y\neq C_{2}\cap Y$ and $|C_{j}\cap \{x_{i},y_{i}\}|=1$ for all $i$ and $j$ with $1\leq i\leq r$ and $j\in \{1,2\}$. We conclude that there is at least one $y_{i}$ in $C_{1}$ such that $y_{i}\notin C_{2}$ and so  $x_{i}$  is in $C_{2}$ but it is not in $C_{1}$. Thus, no two members of $\varphi_{3}$ have more than $r-2$ common elements. It is clear that every member of $\varphi_{4}$ has $(r+1)$-elements and contains no member of $\varphi_{1}\cup\varphi_{2}\cup\varphi_{3}$. By Proposition \ref{pro3}, we conclude that M is a rank-r spike. It is straightforward to show that for every two distinct members of $\mathcal C$, their symmetric difference contains a circuit of $M$. Thus, by Theorem \ref{thm4}, $M$ is a binary spike.
\end{proof}
It is not difficult to check that if $r$ is odd, then the intersection of every two members of $\varphi_{3}$ has odd cardinality and the intersection of every two members of $\varphi_{4}$ has even cardinality and if $r$ is even, then the intersection of every two members of $\varphi_{3}$ has even cardinality and the intersection of every two members of $\varphi_{4}$ has odd cardinality.  Clearly, $|\varphi_{1}|=r$, $|\varphi_{2}|=\frac{r(r-1)}{2}$ and $|\varphi_{3}|=|\varphi_{4}|=2^{r-1}$. Therefore, every rank-r binary spike has $2^r+\frac{r(r+1)}{2}$ circuits. Moreover, $\cap_{i=1}^{r} L_{i}\neq \emptyset$ and $|C\cap \{x_{i},y_{i}\}|=1$ where $1\leq i\leq r$ and $C$ is a member of $\varphi_{3}\cup \varphi_{4}$. 
\section{The $es$-splitting operation on $Z_{r}$}
By applying the $es$-splitting operation on a given matroid with $k$ elements, we obtain a matroid with $k+2$ elements. In this section, our main goal is to give a necessary and sufficient condition for $X\subseteq E(Z_{r})$ with $e\in X$, to obtain $Z_{r+1}$ by applying the $es$-splitting operation on $X$. Now suppose that $M=Z_{r}$ be a binary rank-r spike with the matrix representation $[I_{r}|J_{r}-I_{r}|\textbf{1}]$ over $GF(2)$ whose columns are labeled, in order $x_{1},x_{2},...,x_{r},y_{1},y_{2},...,y_{r},t$. Suppose $\varphi=\varphi_{1}\cup \varphi_{2}\cup \varphi_{3}\cup \varphi_{4}$ be the collection of circuits of $Z_{r}$ defined in section 2. Let $X_{1}=\{x_{1},x_{2},...,x_{r}\}$ and $Y_{1}=\{y_{1},y_{2},...,y_{r}\}$ and let $X$ be a subset of $E(Z_{r})$. By the following lemmas, we give six conditions for membership of $X$ such that, for every element $e$ of this set, $(Z_{r})^e_{X}$ is not the spike $Z_{r+1}$.
\begin{lemma} If $r\geq 4$ and $t\notin X$, then, for every element  $e$ of $X$, the matroid $(Z_{r})^e_{X}$ is not the spike $Z_{r+1}$.
\label{lem6}
\end{lemma}
\begin{proof}
Suppose that $t\notin X$. Without loss of generality, we may assume that there exist $i$ in $\{1,2,...,r\}$ such that $x_{i}\in X$ and $e=x_{i}$. By Proposition \ref{pro1}, the set $\Lambda=\{x_{i},\alpha,\gamma\}$ is a circuit of $(Z_{r})^e_{X}$. Now consider the leg $L_{i}=\{t,x_{i},y_{i}\}$, we have two following cases.\\
\textbf{(i)} If $y_{i}\in X$, then $|L_{i}\cap X|$ is even. By Proposition \ref{pro1}, the leg $L_{i}$ is a circuit of $(Z_{r})^e_{X}$. Now if all other legs of $Z_{r}$ have an odd number of elements of $X$, by Proposition \ref{pro1}, we observe that these legs transform to circuits of cardinality 4 and 5. So there are exactly two 3-circuit in $(Z_{r})^e_{X}$. If not, there is a  $j\neq i$ such that $L_{j}$ is a 3-circuit of $(Z_{r})^e_{X}$ and $(\Lambda\cap L_{i}\cap L_{j})=\emptyset$, we conclude that in each case, for every element  $e$ of $X$, the matroid $(Z_{r})^e_{X}$ is not the spike $Z_{r+1}$. Since $Z_{r+1}$ has $r+1$ legs and the intersection of the legs of $Z_{r+1}$ is non-empty.\\ 
\textbf{(ii)} If $y_{i}\notin X$, then $|L_{i}\cap X|$ is odd. By Proposition \ref{pro1}, $(L_{i}\setminus{x_{i}})\cup \gamma$ is a circuit of $(Z_{r})^e_{X}$. Now if there is the other leg $L_{j}$ such that $|L_{j}\cap X|$ is even, then $L_{j}$ is a circuit of $(Z_{r})^e_{X}$. But $(L_{j}\cap \Lambda\cap ((L_{i}\setminus{x_{i}})\cup \gamma))=\emptyset$, so  $(Z_{r})^e_{X}$ is not the spike $Z_{r+1}$. We conclude that every leg $L_{j}$ with $j\neq i$ has an odd number of elements of $X$. Since $x_{i}\notin L_{j}$, by Proposition \ref{pro1} again, $L_{j}$ is not a 3-circuit in  $(Z_{r})^e_{X}$. Therefore,  $(Z_{r})^e_{X}$ has only two 3-circuits and so, for every element $e$ of $X$, the matroid $(Z_{r})^e_{X}$ is not the spike $Z_{r+1}$.  
\end{proof}
\begin{lemma} If $r\geq 4$ and $e\neq t$, then, for every element  $e$ of $X-t$, the matroid $(Z_{r})^e_{X}$ is not the spike $Z_{r+1}$.
\label{lem7}
\end{lemma}
\begin{proof} 
Suppose that $e\neq t$. Without loss of generality, we may assume that there exist $i$ in $\{1,2,...,r\}$ such that $x_{i}\in X$ and $e=x_{i}$. By Proposition \ref{pro1}, the set $\Lambda=\{x_{i},\alpha,\gamma\}$ is a circuit of $(Z_{r})^e_{X}$ and by Lemma \ref{lem6}, to obtain $Z_{r+1}$, the element $t$ is contained in $X$. Now consider the leg $L_{i}=\{t,x_{i},y_{i}\}$. We have two following cases.\\
\textbf{(i)} If $y_{i}\in X$, then $|L_{i}\cap X|$ is odd. By Proposition \ref{pro1}, $L_{i}\cup \alpha$ and $(L_{i}\setminus{x_{i}})\cup \gamma$ are circuits of $(Z_{r})^e_{X}$. Now if there is the other leg $L_{j}$ such that $|L_{j}\cap X|$ is even, then $L_{j}$ is a circuit of $(Z_{r})^e_{X}$. But $(L_{j}\cap \Lambda\cap ((L_{i}\setminus{x_{i}})\cup \gamma))=\emptyset$, so  $(Z_{r})^e_{X}$ is not the spike $Z_{r+1}$. We conclude that every leg $L_{j}$ with $j\neq i$ has an odd number of elements of $X$. Since $x_{i}\notin L_{j}$, by Proposition \ref{pro1} again, $L_{j}$ is not a 3-circuit in  $(Z_{r})^e_{X}$. Therefore  $(Z_{r})^e_{X}$ has only two 3-circuit and so  $(Z_{r})^e_{X}$ is not the spike $Z_{r+1}$.\\
\textbf{(ii)} If $y_{i}\notin X$, then $|L_{i}\cap X|$ is even. So $L_{i}$ is a circuit of $(Z_{r})^e_{X}$. By similar arguments in Lemma \ref{lem6} (i), one can show that for every element  $e$ of $X-t$, the matroid $(Z_{r})^e_{X}$ is not the spike $Z_{r+1}$.    
\end{proof}
Next by Lemmas \ref{lem6} and \ref{lem7}, to obtain the spike $Z_{r+1}$, we take $t$ in $X$ and $e=t$. 
\begin{lemma} If $r\geq 4$ and there is a circuit $C$  of $\mathcal \varphi_{3}$ such that $|C\cap X|$ is even, then the matroid $(Z_{r})^t_{X}$ is not the spike $Z_{r+1}$.
\label{lem8}
\end{lemma}
\begin{proof} Suppose that $C$ is a circuit of $Z_{r}$ such that is a member of $\varphi_{3}$ and $|C\cap X|$ is even. Then, by Proposition \ref{pro1}, the circuit $C$ is preserved under the $es$-splitting operation. So $C$ is a circuit of $(Z_{r})^t_{X}$. But $|C|=r$. Now if $r> 4$, then $C$ cannot be a circuit of $Z_{r+1}$, since it has no $r$-circuit, and if $r=4$, then, to preserve the members of $\varphi_{2}$ in $Z_{4}$ under the $es$-splitting operation and to have at least one member of $\varphi_{3}$ which has even number of elements of $X$, the set $X$ must be $E(Z_{r})-t$ or $t$. But in each case $(Z_{4})^t_X$ has exactly fourteen 4-circuits, so it is not the spike $Z_{5}$, since this spike has exactly ten 4-circuit. We conclude that the matroid $(Z_{r})^t_{X}$ is not the spike $Z_{r+1}$.  
\end{proof} 
\begin{lemma} If $r\geq 4$ and $|X\cap \{x_{i},y_{i}\}|=2$, for $i$ in $\{1,2,...,r\}$, then the matroid $(Z_{r})^t_{X}$ is not the spike $Z_{r+1}$ unless $r$ is odd and for all $i$, $\{x_{i},y_{i}\}\subset X$, in which case $Z_{r+1}$ has $\gamma$ as a tip.
\label{lem9}
\end{lemma}
\begin{proof}
Suppose that $\{x_{i},y_{i}\}\subset X$ for $i\in\{1,2,...,r\}$. Since $t\in X$ and $e=t$, after applying the $es$-splitting operation, the leg $\{t,x_{i},y_{i}\}$ turns into two circuits $\{t,x_{i},y_{i},\alpha\}$ and $\{x_{i},y_{i},\gamma\}$. Now consider the leg $L_{j}=\{t,x_{j},y_{j}\}$ where $j\neq i$. If  $|L_{j}\cap X|$ is even (this means $\{x_{j},y_{j}\}\nsubseteq X$), then $L_{j}$ is a circuit of $(Z_{r})^e_{X}$. But $\{x_{i},y_{i},\gamma\}\cap \{t,x_{j},y_{j}\}=\emptyset$ and this contradicts the fact that the intersection of the legs of a spike is not the empty set. So $\{x_{j},y_{j}\}$ must be a subset of $X$. We conclude that $\{x_{k},y_{k}\}\subset X$ for all $k\neq i$. Thus $X=E(Z_{r})$. But in this case, $r$ cannot be even since every circuit in $\varphi_{3}$ has even cardinality and by Lemma \ref{lem8}, the matroid $(Z_{r})^t_{X}$ is not the spike $Z_{r+1}$.

Now we show that if $X=E(Z_{r})$, and $r$ is odd, then $(Z_{r})^t_{X}$ is the spike $Z_{r+1}$ with tip $\gamma$.
Clearly, every leg of $Z_{r}$ has an odd number of elements of $X$. Using Proposition 1, after applying the $es$-splitting operation, we have the following changes.\\
For $i$ in $\{1,2,...,r\}$, $L_{i}$ transforms to two circuits $(L_{i}\setminus t)\cup\gamma$ and $L_{i}\cup \alpha$, every member of $\varphi_{2}$ is preserved, and if $C\in \varphi_{3}$, then $C\cup\alpha$ and $C\cup\{t,\gamma\}$ are circuits of $(Z_{r})^t_{X}$. Finally, if $C\in \varphi_{4}$, then $C$ and $(C\setminus t)\cup\{\alpha,\gamma\}$ are circuits of $(Z_{r})^t_{X}$. Note that, since $X=E(M)$ with $e=t$, there are no two disjoint $OX$-circuits in $Z_{r}$ such that their union be minimal. Therefore the collection $\mathcal C_{5}$ in Proposition 1 is empty. Now suppose that $\alpha$ and $t$ play the roles of $x_{r+1}$ and $y_{r+1}$, respectively, and $\gamma$ plays the role of tip. Then   we have the spike $Z_{r+1}$ with tip $\gamma$ whose  collection $\psi$ of circuits is $\psi_{1}\cup\psi_{2}\cup\psi_{3}\cup\psi_{4}$ where
\textit{\begin{itemize}[font=\normalfont\textbf]
\item[]  $\mathcal \psi_{1}=\{(L_{i}\setminus t)\cup \gamma:1\leq i\leq r\}\cup\Lambda$;
\item[]  $\mathcal \psi_{2}=\{\{x_{i},y_{i},x_{j},y_{j}\}: 1\leq i < j \leq r \}\cup\{(L_{i}\cup \alpha: 1\leq i\leq r\}$;
\item[]  $\mathcal \psi_{3}=\{C\cup \alpha: C\in \varphi_{3}\}\cup\{C: C\in \varphi_{4}\}$; 
\item[]  $\mathcal \psi_{4}=\{C\cup \{t,\gamma\}: C\in \varphi_{3}\}\cup\{(C\setminus t)\cup\{\alpha,\gamma\}: C\in \varphi_{4}\}$. 
\end{itemize}}
\end{proof} 
In the following lemma, we shall use the well-known facts that if a matroid $M$ is  $n$-connected with $E(M)\geq 2(n-1)$, then all circuits and all cocircuits of $M$ have at least $n$ elements, and if $A$ is a matrix that represents $M$ over $GF(2)$, then the cocircuit space of $M$ equals the row space of $A$. 
\begin{lemma} If $|X|\leq r$, then the matroid $(Z_{r})^t_{X}$ is not the spike $Z_{r+1}$.
\label{lem10}
\end{lemma}
\begin{proof} Suppose $X\subset E(Z_{r})$ such that $|X|\leq r$. Then, by Lemmas \ref{lem6}, \ref{lem7} and \ref{lem8}, $t\in X$ with $e=t$ and $|X\cap \{x_{i},y_{i}\}|=1$ for all $i$ in $\{1,2,...,r\}$. Therefore, there are at least two elements $x_{j}$ and $y_{j}$ with $i\neq j$ not contained in $X$ and so the leg $L_{j}=\{t,x_{j},y_{j}\}$ has an odd number of elements of $X$. Thus, after applying the $es$-splitting operation $L_{j}$ transforms to $\{x_{j},y_{j},\gamma\}$. Now let $L_{k}=\{t,x_{k},y_{k}\}$ be another leg of $Z_{r}$. If $|L_{k}\cap X|$ is even, then $L_{k}$ is a circuit of $(Z_{r})^t_{X}$. But $(L_{k}\cap \Lambda \cap \{x_{j},y_{j},\gamma\})=\emptyset$. Hence, in this case, the matroid $(Z_{r})^t_{X}$ is not the spike $Z_{r+1}$ . We may now assume that every other leg of $Z_{r}$ has an odd number of elements of $X$. Then, for all $j\neq i$, the elements $x_{j}$ and $y_{j}$ are not contained in $X$. We conclude that $|X|=1$ and in the last row of the matrix that represents the matroid $(Z_{r})^t_{X}$  there are two entries 1 in the  corresponding columns of $t$ and $\alpha$. Hence, $(Z_{r})^t_{X}$ has a 2-cocircuit and it is not the matroid $Z_{r+1}$ since spikes are 3-connected matroids.      
\end{proof}
By Lemmas \ref{lem9} and \ref{lem10}, we must check that if $|x|=r+1$, then, by using the $es$-splitting operation, can we build the spike $Z_{r+1}$? 
\begin{lemma} If $r\geq 4$  and $|X\cap X_{1}|$ be odd, then the matroid $(Z_{r})^t_{X}$ is not the spike $Z_{r+1}$.
\label{lem11}
\end{lemma}
\begin{proof} Suppose that $r$ is even and $|X\cap X_{1}|$ is odd. Since $t\in X$ and $|X|=r+1$, so $|X\cap Y_{1}|$ must be odd. Therefore the set  $X$ must be $C\cup t$ where $C\in \varphi_{3}$. But $|C\cap X|$ is even and by Lemma \ref{lem8}, the matroid $(Z_{r})^t_{X}$ is not the spike $Z_{r+1}$. Now Suppose that $r$ is odd, $r\geq 4$ and $|X\cap X_{1}|$ is odd. Then $|X\cap Y_{1}|$ must be even and so $X=C$ where $C\in \varphi_{4}$. By definition of binary spikes, there is a circuit $C'$ in $\varphi_{4}$ such that $C'=C\Delta\{x_{i},y_{i},x_{j},y_{j}\}$ for all $i$ and $j$ with $1\leq i < j\leq r$. Clearly, $|(E-C')\cap X|=2$. Since $(E-C')$ is a circuit of $Z_{r}$ and is a member of $\varphi_{3}$, by Lemma \ref{lem8}, the matroid $(Z_{r})^t_{X}$ is not the spike $Z_{r+1}$.

\end{proof}
Now suppose that $M$ is a binary rank-r spike with tip $t$ and $r\geq 4$. Let $X\subseteq E(M)$ and $e\in X$ and let $E(M)-E(M^e_{X})=\{\alpha,\gamma\}$ such that $\{e,\alpha,\gamma\}$ is a circuit of $M^e_{X}$. Suppose $\varphi=\cup_{i=0}^{4}\varphi_{i}$ be the collection of circuits of $M$ where $\varphi_{i}$ is defined in section 2.  With these preliminaries, the next two theorems are the main results of this paper.  
\begin{theorem} 
Suppose that $r$ is an even integer  greater than three. Let $M$ be a rank-r binary spike with tip $t$. Then $M^e_{X}$ is a rank-$(r+1)$ binary spike if and only if $X=C$ where $C\in \varphi_{4}$ and $e=t$.
\end{theorem}
\begin{proof}
Suppose that $M=Z_{r}$ and $X\subseteq E(M)$ and $r$ is even. Then, by combining the last six lemmas, $|X|=r+1$; and $X$ contains an even number of elements of $X_{1}$ with $t\in X$. The only subsets of $E(Z_{r})$ with these properties are members of $\varphi_{4}$. Therefore $X=C$ where $C\in \varphi_{4}$ and by Lemma \ref{lem7}, $e=t$. Conversely, let $X=C$ where $C\in \varphi_{4}$. Then, by using Proposition \ref{pro1}, every leg of $Z_{r}$ is preserved under the $es$-splitting operation since they have an even number of elements of $X$. Moreover, for $i\in\{1,2,...,r\}$, every leg $L_{i}$ contains $e$ where $e=t$. So $L_{i}\setminus t$ contains an odd number of elements of $X$ and by Proposition \ref{pro1}, the set $(L_{i}\setminus t)\cup \{\alpha,\gamma\}$ is a circuit of $M^t_{X}$. Clearly, every member of $\varphi_{2}$ is preserved. Now let $C'\in \varphi_{3}$. Then $t\notin C'$. We have two following cases.\\
\textbf{(i)} Let $C'=(E-X)\Delta \{x_{r-1},y_{r-1}\}$. Then $|C'\cap X|=1$ and by Proposition \ref{pro1}, $C'\cup \alpha$ and $C'\cup \{t,\gamma\}$ are circuits of $M^t_{X}$.\\
\textbf{(ii)} Let $C'=(E-C'')\Delta \{x_{r-1},y_{r-1}\}$ where $C''\neq X$ and $C''\in \varphi_{4}$. Since $|X|=r+1$ and $|C''\cap X|$ is odd, the cardinality of the set $X\cap (E-C'')$ is even and so $|C'\cap X|$ is odd. Therefore, by Proposition \ref{pro1} again, $C'\cup \alpha$ and $C'\cup \{t,\gamma\}$ are circuits of $M^t_{X}$.

Evidently, if $C\in \varphi_{4}$, then $|C\cap X|$ is odd and by Proposition \ref{pro1}, $C\cup \alpha$ and $(C\setminus t)\cup \gamma$ are circuits of $M^t_{X}$. Moreover, there are no two disjoint $OX$-circuit in $\varphi$. So the collection $\mathcal C_{5}$ in Proposition 1 is empty.
To complete the proof, suppose that $\alpha$ and $\gamma$ play the roles of $x_{r+1}$ and $y_{r+1}$, respectively, then we have the spike $Z_{r+1}$ with collection of circuits $\psi=\psi_{1}\cup\psi_{2}\cup\psi_{3}\cup\psi_{4}$ where
\textit{\begin{itemize}[font=\normalfont\textbf]
\item[]  $\mathcal \psi_{1}=\{L_{i}=\{t,x_i,y_{i}\}:1\leq i\leq r\}\cup\Lambda$;
\item[]  $\mathcal \psi_{2}=\{\{x_{i},y_{i},x_{j},y_{j}\}: 1\leq i < j \leq r \}\cup\{(L_{i}\setminus t)\cup\{\alpha,\gamma\}:1\leq i\leq r \}$;
\item[]  $\mathcal \psi_{3}=\{C\cup \alpha: C\in \varphi_{3}\}\cup\{(C\setminus t)\cup \gamma: C\in \varphi_{4}\}$; 
\item[]  $\mathcal \psi_{4}=\{C\cup \{t,\gamma\}: C\in \varphi_{3}\}\cup\{C\cup \alpha: C\in \varphi_{4}\}$. 
\end{itemize}}
\end{proof} 
\begin{theorem} 
Suppose that $r$ is an odd integer greater than three. Let $M$ be a rank-r binary spike with tip $t$. Then $M^e_{X}$ is a rank-$(r+1)$ binary spike if and only if $X=C\cup t$ where $C\in \varphi_{3}$ or $X=E(M)$, and $e=t$.
\label{thm13}
\end{theorem}
\begin{proof} Suppose that $M=Z_{r}$ and $X\subseteq E(M)$. Let $X=E(M)$. Then, by Lemma \ref{lem9}, the matroid $M^t_{X}$ is the spike $Z_{r+1}$ with tip $\gamma$. Now, by combining the last six lemmas., $|X|=r+1$ and $X$ contains an even number of elements of $X_{1}$ with $t\in X$. The only subsets of $E(Z_{r})$ with these properties are in  $\{C\cup t: C\in \varphi_{3} \}$. Conversely, let $X=C\cup t$ where $C\in \varphi_{3}$. Clearly, every member of $\varphi_{3}$ contains an odd number of elements of $X$. Now let $C'$ be a member of $\varphi_{4}$. If $C'=E(Z_{r})-C$, then $C'$ contains an odd number of elements of $X$. If $C'\neq E(Z_{r})-C$, then there is a $C''\in \varphi_{3}$ such that $C'=E(Z_{r})-C''$. Therefore $|C\cap C'|=|C\cap(E(Z_{r})-C'')|=|C-(C\cap C'')|$ and so $|C\cap C'|$  is even. So $C'$ contains an odd number of elements of $X$ and, by Proposition \ref{pro1} again $C'\cup \alpha$ and $(C\setminus t)\cup \gamma$ are circuits of $M^t_{X}$. Evidently, if $C_{1}$ and $C_{2}$ be disjoint $OX$-circuits of $Z_{r}$, then one of  $C_{1}$ and $C_{2}$ is in $\varphi_{3}$ and the other is in $\varphi_{4}$ where $C_{2}=E(Z_{r})-C_{1}$, as $C_{1}\cup C_{2}$ is not minimal, it follows by Proposition \ref{pro1} that $\mathcal C_{5}$ is empty. Now if $\alpha$ and $\gamma$ play the roles of $x_{r+1}$ and $y_{r+1}$, respectively. Then $M^t_{X}$ is the spike $Z_{r+1}$ with collection of circuits $\psi=\psi_{1}\cup\psi_{2}\cup\psi_{3}\cup\psi_{4}$ where
\textit{\begin{itemize}[font=\normalfont\textbf]
\item[]  $\mathcal \psi_{1}=\{L_{i}=\{t,x_i,y_{i}\}:1\leq i\leq r\}\cup\Lambda$;
\item[]  $\mathcal \psi_{2}=\{\{x_{i},y_{i},x_{j},y_{j}\}: 1\leq i < j \leq r \}\cup\{(L_{i}\setminus t)\cup\{\alpha,\gamma\}: 1\leq i\leq r\}$;
\item[]  $\mathcal \psi_{3}=\{C\cup \alpha: C\in \varphi_{3}\}\cup\{(C\setminus t)\cup \gamma: C\in \varphi_{4}\}$; 
\item[]  $\mathcal \psi_{4}=\{C\cup \{t,\gamma\}: C\in \varphi_{3}\}\cup\{C\cup \alpha: C\in \varphi_{4}\}$. 
\end{itemize}}  
\end{proof}
\begin{remark} Note that the binary rank-3 spike is the Fano matroid denoted by $F_{7}$. It is straightforward to check that any one of the seven elements of $F_{7}$ can be taken as the tip, and $F_{7}$ satisfies the conditions of Theorem 13 for any tip. So, there are exactly 35 subset $X$ of $E(F_{7})$ such that $(F_{7})^e_{X}$ is the binary 4-spike where $e$ is a tip of it. Therefore, by Theorem \ref{thm13}, these subsets are $X=E(F_{7})$ for every element $e$ of $X$ and $C\cup z$ for every element $z$ in $E(F_{7})$ not contained in $C$ with $e=z$ where $C$ is a 3-circuit of $F_{7}$.  
\end{remark}


\begin{thebibliography}{00}
\bibitem{1} H. Azanchiler, \emph{Some new operations on matroids and related results}, Ph. D. Thesis,
University of Pune (2005).
\bibitem{2} H. Azanchiler, \emph{Extension of Line-Splitting operation from graphs to binary matroid},
Lobachevskii J. Math. \textbf{24} (2006), 3-12.
\bibitem{3} J. G. Oxley, \emph{Matroid Theory}, Oxford University Press, Oxford (1992).
\bibitem{4} S. B. Dhotre, P. P. Malavadkar and M. M. Shikare, \emph{On 3-connected es-splitting binary
matroids}, Asian-European J. Math. \textbf{9} (1)(2016), 1650017-26.
\bibitem{5} Z. Wu, \emph{On the number of spikes over finite fields}, Discrete Math, \textbf{265} (2003) 261-296
\end{thebibliography}
\end{document}